\documentclass[12pt,reqno]{amsart}

\usepackage{amssymb}

\numberwithin{equation}{section}

\newtheorem{theorem}{Theorem}[section]
\newtheorem{proposition}[theorem]{Proposition}
\newtheorem{lemma}[theorem]{Lemma}
\newtheorem{corollary}[theorem]{Corollary}

\theoremstyle{definition}

\begin{document}

\baselineskip=15pt

\title[Direct image of structure sheaf and parabolic stability]{Direct image of
structure sheaf and parabolic stability}

\author[I. Biswas]{Indranil Biswas}

\address{Department of Mathematics, Shiv Nadar University, NH91, Tehsil
Dadri, Greater Noida, Uttar Pradesh 201314, India}

\email{indranil.biswas@snu.edu.in, indranil29@gmail.com}

\author[M. Kumar]{Manish Kumar}

\address{Statistics and Mathematics Unit, Indian Statistical Institute,
Bangalore 560059, India}

\email{manish@isibang.ac.in}

\author[A. J. Parameswaran]{A. J. Parameswaran}

\address{Kerala School of Mathematics, Kunnamangalam PO, Kozhikode, Kerala, 673571, India}

\email{param@ksom.res.in}



\date{}

\begin{abstract}
Let $f\, :\, X\, \longrightarrow\, Y$ be a dominant generically smooth morphism between irreducible
smooth projective curves over an algebraically closed field $k$ such that ${\rm Char}(k)\,>\,
\text{degree}(f)$ if the characteristic of $k$ is nonzero. We prove that $(f_*{\mathcal O}_X)/{\mathcal O}_Y$
equipped with a natural parabolic structure is parabolic polystable. Several conditions are given
that ensure that the parabolic vector bundle $(f_*{\mathcal O}_X)/{\mathcal O}_Y$ is actually parabolic stable.
\end{abstract}

\maketitle

\section{Introduction}

This work is inspired by \cite{CLV}. We briefly recall the set-up and the results of \cite{CLV}.
Let
$$
f\,\,:\,\, X\,\, \longrightarrow\,\, Y
$$
be a dominant generically smooth morphism between irreducible smooth projective curves over
an algebraically closed field $k$. We assume that
${\rm Char}(k)\, >\, {\rm degree}(f)$ if ${\rm Char}(k)\, >\, 0$. Then the natural short exact
sequence of vector bundles on $Y$
$$
0\, \longrightarrow\,{\mathcal O}_Y \, \longrightarrow\, f_*{\mathcal O}_X \, \longrightarrow\,
(f_*{\mathcal O}_X)/{\mathcal O}_Y \, \longrightarrow\,0
$$
splits. In this set-up, \cite{CLV} gives numerous conditions under which $(f_*{\mathcal O}_X)/{\mathcal O}_Y$
is semistable or stable. Here are the main results of \cite{CLV}:
\begin{enumerate}
\item If $f$ factors nontrivially as $f\, =\, g\circ h$, where $h$ is \'etale and the homomorphism
of \'etale fundamental groups induced by $g$ is surjective (such a map is
called primitive), then $(f_*{\mathcal O}_X)/{\mathcal O}_Y$
is not semistable.

\item If $f$ is \'etale, then $(f_*{\mathcal O}_X)/{\mathcal O}_Y$ is semistable, and
$(f_*{\mathcal O}_X)/{\mathcal O}_Y$ is stable if the representation of the Galois group
$\text{Gal}(X/Y)$ on the fiber of $f$ over an unbranched point of $Y$ is irreducible.

\item Let $f$ be a general primitive covering of degree $r$, where
$\text{genus}(Y)\,=\, h$. Then $(f_*{\mathcal O}_X)/{\mathcal O}_Y$ is semistable if
$h\,=\,1$, and $(f_*{\mathcal O}_X)/{\mathcal O}_Y$ is stable if $h\, \geq\, 2$.
\end{enumerate}

The vector bundle $f_*{\mathcal O}_X$ captures the geometry of the covering map $f$ in an essential
way (cf. \cite{BP}).

The above vector bundle $(f_*{\mathcal O}_X)/{\mathcal O}_Y$ has a natural parabolic structure. This
parabolic structure can be roughly explained as follows. Assume that $\psi\, :\, Z\, \longrightarrow\, Y$
is a ramified Galois cover such that the natural map
$$
\widetilde{f}\ :\ \widetilde{Z\times_Y X}\ \longrightarrow\ Z
$$
is \'etale, where $\widetilde{Z\times_Y X}$ is the normalizer of the fiber product $Z\times_Y X$.
Then the vector bundle $(\widetilde{f}_*{\mathcal O}_{\widetilde{Z\times_Y X}})/{\mathcal O}_Z$ is an equivariant vector
bundle over $Z$ for the action of the Galois group $\text{Gal}(\psi)$ on $Z$. Using the correspondence between
parabolic vector bundles and equivariant vector bundles, this equivariant vector bundle on $Z$ produces
a parabolic vector bundle on $Y$. This parabolic vector bundle on $Y$ coincides with $(f_*{\mathcal O}_X)/{\mathcal O}_Y$
equipped with the parabolic structure.

We prove that $(f_*{\mathcal O}_X)/{\mathcal O}_Y$ equipped with the parabolic structure is always
parabolic polystable (see Theorem \ref{thm1}). This generalizes the results of \cite{CLV} to a wider context.

We then give several conditions under which $(f_*{\mathcal O}_X)/{\mathcal O}_Y$ equipped with the
parabolic structure is actually parabolic stable.

If the action of the Galois group $\Gamma$ of the Galois closure of $f$ on the space of functions on
a fiber of $f$ over an unbranched point of $f$ modulo scalars is irreducible, then
$(f_*{\mathcal O}_X)/{\mathcal O}_Y$ equipped with the parabolic structure is parabolic stable
(see Proposition \ref{prop2}).

If the Galois group $\Gamma$ is the symmetric group $S_r$ or the alternating group $A_r$,
where $r$ is the degree of $f$, then $(f_*{\mathcal O}_X)/{\mathcal O}_Y$ equipped with the
parabolic structure is parabolic stable (see Theorem \ref{thm2}).

If $f$ is genuinely ramified (same as primitive) and Morse (see \cite{BKP}), then
$(f_*{\mathcal O}_X)/{\mathcal O}_Y$ equipped with the
parabolic structure is parabolic stable (see Corollary \ref{cor1}).
Note that a general genuinely ramified map is Morse. Hence the parabolic analogue of the main
result (3) of \cite{CLV} holds. Also note that for a given genuinely ramified map $f$, it is possible to check whether it is Morse or not. So parabolic stability can be decided for $f$ while stability is existential.

\section{Direct image and parabolic structure}\label{se2}

\subsection{Direct image of structure sheaf}\label{se2.1}

The base field $k$ is algebraically closed. Let $X$ and $Y$ be irreducible smooth projective curves
defined over $k$ and
\begin{equation}\label{e1}
f\,\,:\,\, X\,\, \longrightarrow\,\, Y
\end{equation}
a dominant generically smooth morphism such that 
$${\rm Char}(k)\, >\, r\,:=\, \text{degree}(f)$$ if ${\rm Char}(k)\,\not=\, 0$. Consider the vector bundle
$f_*{\mathcal O}_X \,\longrightarrow\, Y$ of rank $r$. There is a natural homomorphism $\alpha\, :\, {\mathcal O}_Y\,
\longrightarrow\,f_*{\mathcal O}_X$ given by the identity map
$f^*{\mathcal O}_Y\,=\, {\mathcal O}_X\, \longrightarrow\, {\mathcal O}_X$ \cite[p.~110]{Ha}. On the other hand, there is
the trace morphism $\beta\, :\, f_*{\mathcal O}_X \, \longrightarrow\, {\mathcal O}_Y$ which is constructed using
addition of functions. The
composition $\beta\circ\alpha$ coincides with $r\cdot {\rm Id}_{{\mathcal O}_Y}$, and hence we have the decomposition
\begin{equation}\label{e2}
f_*{\mathcal O}_X \,\,=\,\, {\mathcal O}_Y \oplus \text{kernel}(\beta)
\end{equation}
(recall that ${\rm Char}(k)\, >\, r$ if it is nonzero). For convenience, denote
$E\,\, :=\,\, \text{kernel}(\beta)\,\, \subset\,\, f_*{\mathcal O}_X$,
so \eqref{e2} gives
\begin{equation}\label{e3}
f_*{\mathcal O}_X \,\,=\,\, {\mathcal O}_Y \oplus E.
\end{equation}

The vector bundle $E$ has a natural parabolic structure which will be briefly recalled.

\subsection{Parabolic bundles}\label{se2.2}

As before, $X$ is an irreducible smooth projective curve over $k$. Let
$$
D\,\,:=\,\, \{x_1,\, \cdots,\, x_\ell\}\,\, \subset\,\, X
$$
be a finite subset. Take a vector bundle $V$ on $X$. A quasi-parabolic structure on $V$ is a
strictly decreasing filtration of subspaces
\begin{equation}\label{e4}
V_{x_i}\,=\, V^1_i\, \supsetneq\, V^2_i \,\supsetneq\, \cdots\, \supsetneq\,
V^{n_i}_i \, \supsetneq\, V^{n_i+1}_i \,=\, 0
\end{equation}
for every $1\, \leq\, i\, \leq\, \ell$; here $V_{x_i}$ denotes the fiber
of the vector bundle $V$ over the point $x_i\,\in\, D$. A \textit{parabolic structure} on $V$ is a
quasi-parabolic structure as above together with $\ell$ increasing sequences of rational numbers
\begin{equation}\label{e5}
0\, \leq\, \alpha_{i,1}\, <\, \alpha_{i,2}\, <\,
\cdots\, < \, \alpha_{i,n_i}\, < 1\, , \ \ 1\, \leq\, i\, \leq\, \ell \, ;
\end{equation}
the rational number $\alpha_{i,j}$ is called the parabolic weight of the subspace $V^j_i$ in
the quasi-parabolic filtration in \eqref{e4}. The multiplicity of a parabolic weight $\alpha_{i,j}$ at
$x_i$ is defined to be the dimension of the $k$-vector space $V^j_i/V^{j+1}_i$.
A parabolic vector bundle is a vector bundle equipped with a parabolic structure.
The subset $D$ is called the parabolic divisor. (See \cite{MS}, \cite{MY}.)

Let $V_*$ denote the parabolic vector bundle defined by \eqref{e4} and \eqref{e5}. The parabolic degree of
$V_*$ is defined to be
$$
\text{par-deg}(V_*)\,\,:=\,\, \text{degree}(V) +\sum_{i=1}^\ell\sum_{j=1}^{n_i} \alpha_{i,j}
\cdot\dim V^j_i/V^{j+1}_i,
$$
and the parabolic weight of $V_*$ is defined to be $\sum_{i=1}^\ell\sum_{j=1}^{n_i} \alpha_{i,j}
\cdot\dim V^j_i/V^{j+1}_i$.
The rational number $$\mu(V_*)\,:=\, \frac{\text{par-deg}(V_*)}{{\rm rank}(V)}$$ is called the parabolic
slope of $V_*$.

Any subbundle $W\, \subset\, V$ has a parabolic structure induced by the parabolic structure of $V_*$; the
resulting parabolic bundle will be denoted by $W_*$. The parabolic
vector bundle $V_*$ is called \textit{stable} (respectively, \textit{semistable}) if
$$
\mu(W_*)\, <\, \mu(V_*) \ \ \ \, {\rm (}\text{respectively, }\ \ \mu(W_*)\, \leq\, \mu(V_*){\rm )}
$$
for all subbundle $0\, \not=\, W\, \subsetneq\, V$. The parabolic
vector bundle $V_*$ is called \textit{polystable} if it is a direct sum of stable parabolic bundles of same parabolic slope.
(See \cite{MS}, \cite{MY}.)

\subsection{Parabolic structure on a direct image}\label{se2.3}

Take $f$ as in \eqref{e1}. We will show that the vector bundle $E$ in \eqref{e3} has a natural parabolic 
structure. This will be done by first constructing a natural parabolic structure on the direct image 
$f_*{\mathcal O}_X$.

Let
\begin{equation}\label{e6}
R\, \subset\, X
\end{equation}
be the ramification locus of $f$. The given condition that ${\rm Char}(k)\, >\, r$, if it
is nonzero, implies that $f$ is
tamely ramified. For any point $x\,\in\, X$, let $m_x\, \geq\, 1$ be the multiplicity
of $f$ at $x$, so we have $m_x\, \geq\, 2$ if and only if $x\,\in\, R$. Let
\begin{equation}\label{e7}
\Delta\,\,=\,\, f (R) \,\, \subset\,\, Y
\end{equation}
be the branch locus for $f$, which will be the parabolic divisor for the parabolic structure on $f_*{\mathcal O}_X$.

For any point $y\, \in\, Y$, the fiber $(f_*{\mathcal O}_X)_y$ of the vector bundle $f_*{\mathcal O}_X$ over $y$
has a decomposition
\begin{equation}\label{e8}
(f_*{\mathcal O}_X)_y\,=\, \bigoplus_{x\,\in \,f^{-1}(y)} F_x
\end{equation}
such that $\dim F_x\,=\, m_x$, where $m_x$ is defined above (see \cite[p.~19562, (4.4)]{AB}); here
$f^{-1}(y)$ denotes the set-theoretic inverse image.
To describe the subspace $F_x\, \subset\, (f_*{\mathcal O}_X)_y$ in \eqref{e8}, consider the homomorphism
\begin{equation}\label{e9}
f_* \left({\mathcal O}_X\left(-\sum_{z \in f^{-1}(y)\setminus x} m_z z\right)\right)\,\, \longrightarrow\,\,
f_*{\mathcal O}_X
\end{equation}
given by the natural inclusion of ${\mathcal O}_X\left(-\sum_{z \in f^{-1}(y)\setminus x} m_z z\right)$ in
${\mathcal O}_X$. The subspace $F_x\,\subset\, (f_*{\mathcal O}_X)_y$ in \eqref{e8} is the image of the
homomorphism of fibers
$$
f_* \left({\mathcal O}_X\left(-\sum_{z \in f^{-1}(y)\setminus x} m_z z\right)\right)_y\,\,
\longrightarrow\,\, (f_*{\mathcal O}_X)_y
$$
corresponding to the homomorphism of coherent sheaves in \eqref{e9}.

The parabolic structure on $(f_*{\mathcal O}_X)_y$ will be described by giving a parabolic structure on each
direct summand $F_x$ in \eqref{e8} and then taking their direct sum. To give a parabolic structure on $F_x$, first
note that for any $j\, \geq\, 0$, there is a natural injective homomorphism of coherent sheaves
\begin{equation}\label{e10}
f_*\left( {\mathcal O}_X\left(-jx -\sum_{z \in f^{-1}(y)\setminus x}m_z z \right)\right)\,\, \longrightarrow\,\, f_*{\mathcal O}_X
\end{equation} (see \eqref{e9}).
The image of the fiber $f_*\left( {\mathcal O}_X\left(-jx -\sum_{z \in f^{-1}(y)\setminus x}m_z z \right)\right)_y$ in $(f_*{\mathcal O}_X)_y$ by the homomorphism in \eqref{e10} will be denoted by $\mathbf{F}(x,j)$.
We have a filtration of subspaces of $F_x$:
\begin{equation}\label{e13}
F_x\,:=\, \mathbf{F}(x,0) \, \supset \, \mathbf{F}(x,1) \, \supset \, \mathbf{F}(x,2) \,
\supset\, \cdots\, \supset\,\mathbf{F}(x, m_x-1)\, \supset\, \mathbf{F}(x, m_x) \,=\,0.
\end{equation}
Note that
$$
f_* \left( {\mathcal O}_X\left(-\sum_{z \in f^{-1}(y)}m_z z\right) \right)\,\,=\,\, (f_*{\mathcal O}_X)\otimes {\mathcal O}_Y(-y)
$$
by the projection formula, and hence we have $\mathbf{F}(x,m_x) \,=\,0$. The quasiparabolic filtration
of $F_x$ is the one obtained in \eqref{e13}. The parabolic weight of the subspace
$\mathbf{F}(x,j)\, \subset\, F_x$
in \eqref{e13} is $\frac{j}{m_x}$. (See \cite{AB}.) The parabolic structure on $(f_*{\mathcal O}_X)_y$ is given by the
parabolic structure on the subspaces $F_x$ using the decomposition in \eqref{e8}. More precisely, for any $0\,
\leq\, \lambda\, <\, 1$,
the subspace of $(f_*{\mathcal O}_X)_y$ of parabolic weight at least $\lambda$ is the direct sum of subspaces of $F_x$,
$x\,\in \,f^{-1}(y)$, of parabolic weight at least $\lambda$.

Let $(f_*{\mathcal O}_X)_*$ denote the parabolic vector bundle of rank $r$ defined by the above parabolic 
structure on $f_*{\mathcal O}_X$. From the construction of the parabolic structure of $(f_*{\mathcal O}_X)_*$ 
it follows immediately that the induced parabolic structure on the direct summand ${\mathcal O}_Y$ in \eqref{e3} 
is the trivial one (it has no nonzero parabolic
weight). Consider the parabolic structure on the direct summand $E$ in \eqref{e3} induced by the 
parabolic structure of $(f_*{\mathcal O}_X)_*$. The resulting parabolic bundle will be denoted by $E_*$.
Note that we have
\begin{equation}\label{e14}
(f_*{\mathcal O}_X)_*\,\,=\,\, {\mathcal O}_Y\oplus E_*;
\end{equation}
In Section \ref{se3.2} we will show that $E_*$ is parabolic polystable.

\section{Parabolic pullback and polystability of direct image}\label{se3}

\subsection{Pullback of parabolic bundles}\label{se3.1}

Let $M$ and $N$ be irreducible smooth projective curves defined over $k$, and let
\begin{equation}\label{vp}
\varphi\, :\, M\, \longrightarrow\, N
\end{equation}
be a generically smooth tamely ramified covering map. Take any
parabolic vector bundle $W_*$ on $N$. Then we have the pulled back parabolic vector bundle $\varphi^*W_*$
on $M$; see \cite[Section~3]{AB} for its construction. If $D$ is the parabolic divisor for $W_*$, then the
parabolic divisor of $\varphi^*W_*$ is the reduced divisor $\varphi^{-1}(D)_{\rm red}$. To give an idea
of the pullback $\varphi^*W_*$, let $z\,\in\, D$ be a parabolic point of $W_*$. Assume that the
quasiparabolic filtration at $z$ is 
\begin{equation}\label{f1}
W_{z}\,=\, W_1\, \supsetneq\, W_2 \,\supsetneq\, \cdots\, \supsetneq\,
W_n \, \supsetneq\, W_{n+1} \,=\, 0
\end{equation}
(see \eqref{e4}), and the parabolic weights are
$$
0\, \leq\, \alpha_{1}\, <\, \alpha_{2}\, <\, \cdots\, < \, \alpha_{n}\, < 1
$$
(see \eqref{e5}). Then ``roughly'' the underlying vector bundle for the parabolic
bundle $\varphi^*W_*$ is $\varphi^*W$, where
$W$ is the vector bundle underlying $W_*$, and the quasiparabolic filtration of $\varphi^*W_*$ at any
$y\, \in\, \varphi^{-1}(z)_{\rm red}$ is ``roughly'' the same as in \eqref{f1} (since $(\varphi^*W)_y\,=\, W_z$,
the filtration in \eqref{f1} gives a filtration of subspaces of $(\varphi^*W)_y$). The parabolic
weights are ``roughly''
$$
0\, \leq\, m\cdot\alpha_{1}\, <\, m\cdot\alpha_{2}\, <\, \cdots\, < \, m\cdot\alpha_{n},
$$
where $m$ is the multiplicity of $\varphi$ at $y$. But now the problem is that some $m\cdot\alpha_i$ might exceed
$1$. Those weights are decreased by $1$ and it is compensated by performing a forward elementary transformation
(also called Hecke transformation) on the corresponding subspace of the filtration. This process is
repeated enough number of times until all the parabolic weights are brought under $1$.

To explain this more concretely,
we consider the example of parabolic line bundles. Take a line bundle $L$ on $N$. Set the parabolic points
to be $\{x_1,\, \cdots,\, x_\ell\}\, \in\, N$, and equip $L$ with the parabolic weight $0\, <\, \alpha_i\,
<\, 1$ at $x_i$; the resulting parabolic line bundle will be denoted by $L_*$. Let $\{y_{i,1},\, \cdots\,
y_{i, b_i}\}\, =\, \varphi^{-1}(x_i)_{\rm red}\, \subset\, M$ be the reduced inverse image. The multiplicity
of $\varphi$ at $y_{i,j}$ is $m_{i,j}$. Then the underlying line bundle for the pullback $\varphi^*L_*$ is
$$
(\varphi^*L)\otimes {\mathcal O}_M\left(\sum_{i=1}^\ell \sum_{j=1}^{b_i}\lfloor m_{i,j}\cdot\alpha_i\rfloor
\cdot y_{i,j}\right),
$$
where $\lfloor{\lambda}\rfloor$ denotes the integral part of $\lambda$,
and the parabolic weight at $y_{i,j}$ is $m_{i,j}\cdot\alpha_i - \lfloor m_{i,j}\cdot\alpha_i\rfloor$. (See
\cite[p.~19554]{AB}.) Note that tensoring with ${\mathcal O}_M(\lfloor m_{i,j}\cdot\alpha_i\rfloor \cdot y_{i,j})$
is same as tensoring with ${\mathcal O}_M(y_{i,j})$ iteratively $\lfloor m_{i,j}\cdot\alpha_i\rfloor$ times.

We have
\begin{equation}\label{e17a}
\text{par-deg}(\varphi^* W_*)\,\,=\,\, \text{degree}(\varphi)\cdot \text{par-deg}(W_*)
\end{equation}
(see \cite[p.~19560, Lemma 3.5(1)]{AB}).

Now assume that that $\varphi$ in \eqref{vp} is a ramified Galois covering such that the order of the
Galois group $\text{Gal}(\varphi)$ is coprime to ${\rm Char}(k)$ if ${\rm Char}(k)$
is nonzero. Consider the parabolic vector bundle
$({\varphi}_*{\mathcal O}_M)_*$ constructed in Section \ref{se2.3}. Then
\begin{equation}\label{e17}
\varphi^*(\varphi_*{\mathcal O}_M)_*\,\,=\,\, {\mathcal O}_M\otimes_k k[\text{Gal}(\varphi)]
\end{equation}
(see \cite[p.~19566, Proposition 4.2(2)]{AB}). In particular, the parabolic structure on
$\varphi^*(\varphi_*{\mathcal O}_M)_*$ is the trivial one (there are no nonzero parabolic weights).

\subsection{Polystability of direct image}\label{se3.2}

Take $f$ as in \eqref{e1} and consider the parabolic vector bundle $(f_*{\mathcal O}_X)_*$
constructed in Section \ref{se2.3}.

\begin{lemma}\label{lem1}
The parabolic degree of the parabolic vector bundle $(f_*{\mathcal O}_X)_*$ is zero.
\end{lemma}

\begin{proof}
Take a ramification point $x\, \in\, X$ of $f$. The contribution of $x$ to the parabolic weight of
$(f_*{\mathcal O}_X)_*$ is
\begin{equation}\label{e15}
\frac{1}{m_x}(1 +2+ \ldots + (m_x-1))\,\,=\,\, \frac{m_x-1}{2}.
\end{equation}
On the other hand,
$$
\chi(f_*{\mathcal O}_X) \,=\, \chi({\mathcal O}_X) \,=\, 1- g_X,
$$
where $g_X$ is the genus of $X$. So by Riemann--Roch for $f_*{\mathcal O}_X$,
$$
\text{degree}(f_*{\mathcal O}_X) \,=\, 1- g_X - r(1-g_Y)\,=\, r(g_Y-1)-g_X+1,
$$
where $g_Y$ is the genus of $Y$ and $r\,=\, {\rm degree}(f)$. Now, by Riemann--Hurwitz formula, 
\begin{equation}\label{e16}
\text{degree}(f_*{\mathcal O}_X) \,\,=\,\, r(g_Y-1)-g_X+1 \,\,=\,\, - \frac{1}{2}\sum_{x\in X} (m_x-1).
\end{equation}
Combining \eqref{e15} and \eqref{e16} it follows that $\text{par-deg}((f_*{\mathcal O}_X)_*)\,=\, 0$.
\end{proof}

\begin{proposition}\label{prop1}
The parabolic vector bundle $(f_*{\mathcal O}_X)_*$ is polystable.
\end{proposition}

\begin{proof}
Let 
\begin{equation}\label{e18}
\phi\,\, :\,\, M\,\, \longrightarrow\,\, Y
\end{equation}
be the Galois closure of $f$ and $\Gamma\,=\, {\rm Gal}(\phi)$ the Galois group. Let
\begin{equation}\label{e19}
\delta\,\, :\,\, M\,\, \longrightarrow\,\, X
\end{equation}
be the natural morphism, so $\phi\,=\, f\circ\delta$. Note that $\delta$ is also a Galois covering with
${\rm Gal}(\delta)$ a subgroup of $\Gamma$. Also since ${\rm degree}(f)\,=\,r$, it follows that
$\Gamma$ is a subgroup of the symmetric group $S_r$. The degree of $\phi$ divides $r!$, so the given
condition that ${\rm Char}(k)\, >\, r$ (if it is nonzero) implies that
$\text{degree}(\phi)$ is coprime to ${\rm Char}(k)$ (if it is
nonzero). Since $\phi$ in \eqref{e18} is Galois, this
implies that $\phi$ is tamely ramified.

We have a natural inclusion map $f_*{\mathcal O}_X\, \hookrightarrow\, \phi_*{\mathcal O}_M$. In fact,
$f_*{\mathcal O}_X$ is a direct summand of $\phi_*{\mathcal O}_M$. To see this, consider the trace morphism
\begin{equation}\label{e20}
h\,\, :\,\, \phi_*{\mathcal O}_M \,\, \longrightarrow\,\, f_*{\mathcal O}_X.
\end{equation}
The composition of maps
$$
f_*{\mathcal O}_X\,\, \hookrightarrow\,\, \phi_*{\mathcal O}_M \,\, \longrightarrow\,\, f_*{\mathcal O}_X
$$
coincides with multiplication by $\text{degree}(\delta)$ (see \eqref{e19}). Since $\text{degree}(\phi)$ is
coprime to ${\rm Char}(k)$ (if it is nonzero), we know that $\text{degree}(\delta)$ is also coprime to
${\rm Char}(k)$ (if it is nonzero). So
multiplication by $\text{degree}(\delta)$ is an automorphism. Hence we have
\begin{equation}\label{e21}
\phi_*{\mathcal O}_M \,\,=\, f_*{\mathcal O}_X \oplus \text{kernel}(h),
\end{equation}
where $h$ is the homomorphism in \eqref{e20}.

In view of \eqref{e21}, from the construction of parabolic structures on $f_*{\mathcal O}_X$ and 
$\phi_*{\mathcal O}_M$ it follows immediately that $(f_*{\mathcal O}_X)_*$ is a parabolic subbundle of 
$(\phi_*{\mathcal O}_M)_*$. From this it can be deduced that the parabolic pullback $\phi^*(f_*{\mathcal
O}_X)_*$ is a parabolic subbundle of $\phi^*(\phi_*{\mathcal O}_M)_*$. Indeed, this is a straight-forward
consequence of the construction of the pullback of parabolic bundles. From \eqref{e17} we know
that $\phi^*(\phi_*{\mathcal O}_M)_*$ does not have any nonzero parabolic weight. Consequently, its
parabolic subbundle $\phi^*(f_*{\mathcal O}_X)_*$ does not have any nonzero parabolic weight.

We note that Lemma \ref{lem1} and \eqref{e17a} combine together to imply that
\begin{equation}\label{e22}
\text{degree}(\phi^*(f_*{\mathcal O}_X)_*)\,\,=\,\, \text{par-deg}(\phi^*(f_*{\mathcal O}_X)_*)\,\,=\,\, 0
\end{equation}
(recall that $\phi^*(f_*{\mathcal O}_X)_*$ does not have any nonzero parabolic weight). From
\eqref{e17} it follows that $\phi^*(\phi_*{\mathcal O}_M)_*$ is a polystable vector bundle of degree zero. Since
$\text{degree}(\phi^*(f_*{\mathcal O}_X)_*)\,=\,0\,=\, \text{degree}(\phi^*(\phi_*{\mathcal O}_M)_*)$
(see \eqref{e22}), and
$\phi^*(f_*{\mathcal O}_X)_*$ is a subsheaf of the polystable vector bundle $\phi^*(\phi_*{\mathcal O}_M)_*$,
we conclude that $\phi^*(f_*{\mathcal O}_X)_*$ is a polystable vector bundle of degree zero. In particular,
$\phi^*(f_*{\mathcal O}_X)_*$ is semistable.

We will first show that the parabolic vector bundle $(f_*{\mathcal O}_X)_*$ is semistable. Assume that
$(f_*{\mathcal O}_X)_*$ is not parabolic semistable. Let
$$
W_*\,\, \subset\,\, (f_*{\mathcal O}_X)_*
$$
be a parabolic subbundle that violates the inequality for the semistability, meaning
$$
\text{par-deg}(W_*)\,\, >\,\, \text{par-deg}((f_*{\mathcal O}_X)_*)\,\,=\,\, 0
$$
(see Lemma \ref{lem1} for the above equality). Therefore, from \eqref{e17a} it follows immediately
that
$$
\text{par-deg}(\phi^*W_*)\,\, >\,\, \text{par-deg}(\phi^*(f_*{\mathcal O}_X)_*).
$$
Since this contradicts the above observation that $\phi^*(f_*{\mathcal O}_X)_*$ is semistable, we conclude
that $(f_*{\mathcal O}_X)_*$ is parabolic semistable.

Since $(f_*{\mathcal O}_X)_*$ is parabolic semistable, it has a maximal parabolic polystable subbundle
of same parabolic slope which is known as the socle (see \cite[p.~23, Lemma 1.5.5]{HL}). Let
\begin{equation}\label{e23}
S_*\,\, \subset\,\,(f_*{\mathcal O}_X)_*
\end{equation}
be the socle of $(f_*{\mathcal O}_X)_*$. Consider the parabolic subbundle
\begin{equation}\label{e24}
\phi^*S_*\,\, \subset\,\, \phi^* (f_*{\mathcal O}_X)_* .
\end{equation}
Note that $\phi^*S_*$ does not have any nonzero parabolic weights because $\phi^* (f_*{\mathcal O}_X)_*$
does not have any nonzero parabolic weights. Let
\begin{equation}\label{e25}
\textbf{q}\,\, :\,\, \phi^* (f_*{\mathcal O}_X)_*\,\, \longrightarrow\,\,
(\phi^* (f_*{\mathcal O}_X)_*)/\phi^*S_*
\end{equation}
be the quotient map. Since $\text{par-deg}(S_*)\,=\, 0$, from \eqref{e17a} it
follows that
$$
\text{degree}(\phi^*S_*)\,\,=\,\, \text{par-deg}(\phi^*S_*)\,\,=\,\, 0;
$$
note that $(\phi^* (f_*{\mathcal O}_X)_*)/\phi^*S_*$ is locally free because
$(\phi^* (f_*{\mathcal O}_X)_*)$ is semistable of degree zero.
The action of the Galois group $\Gamma$ on $\phi^* (f_*{\mathcal O}_X)_*$ evidently
preserves the subbundle $\phi^* S_*$. So the action of $\Gamma$ on $\phi^* (f_*{\mathcal O}_X)_*$
produces an action of $\Gamma$ on $(\phi^* (f_*{\mathcal O}_X)_*)/\phi^*S_*$.

As $\phi^*(f_*{\mathcal O}_X)_*$ is a polystable vector bundle of degree zero, and
$\text{degree}(\phi^*S_*)\, =\, 0$, the subbundle
$\phi^*S_*\,\subset\, \phi^* (f_*{\mathcal O}_X)_*$ admits a complement. In other words,
there is a homomorphism
\begin{equation}\label{eta}
\eta\,\,:\,\, (\phi^* (f_*{\mathcal O}_X)_*)/\phi^*S_*\,\, \longrightarrow\,\,
\phi^* (f_*{\mathcal O}_X)_*
\end{equation}
such that
\begin{equation}\label{e26}
\textbf{q}\circ\eta\,=\, {\rm Id}_{(\phi^* (f_*{\mathcal O}_X)_*)/\phi^*S_*},
\end{equation}
where $\textbf{q}$ is the projection in \eqref{e25}. The action of $\Gamma\,=\,{\rm Gal}(\phi)$ 
(see \eqref{e18}) on $\phi^* (f_*{\mathcal O}_X)_*$ preserves the subsheaf $\phi^*S_*$ in
\eqref{e24}. Hence the action of $\Gamma$ on $\phi^* (f_*{\mathcal O}_X)_*$ produces an action
of $\Gamma$ on the above quotient $(\phi^* (f_*{\mathcal O}_X)_*)/\phi^*S_*$.
Now using the actions of the Galois group $\Gamma$ on
$\phi^* (f_*{\mathcal O}_X)_*$ and $(\phi^* (f_*{\mathcal O}_X)_*)/\phi^*S_*$, construct the following
homomorphism from $\eta$ in \eqref{eta}:
$$
\widehat{\eta}\,\,:=\,\, \frac{1}{\# \Gamma} \sum_{\gamma\in \Gamma} (\gamma^{-1})'\circ \eta\circ\gamma'' 
\,\, :\,\ \phi^* (f_*{\mathcal O}_X)_*/\phi^*S_*\,\, \longrightarrow\,\,
\phi^* (f_*{\mathcal O}_X)_*,
$$
where $(\gamma^{-1})'$ (respectively, $\gamma''$) denotes the automorphism of $\phi^* (f_*{\mathcal O}_X)_*$
(respectively, $\phi^* (f_*{\mathcal O}_X)_*/\phi^*S_*$) given by the action
of $\gamma^{-1}$ (respectively, $\gamma$) on $\phi^* (f_*{\mathcal O}_X)_*$
(respectively, $\phi^* (f_*{\mathcal O}_X)_*/\phi^*S_*$); recall that the order of $\Gamma$ is coprime to
${\rm Char}(k)$ if ${\rm Char}(k)$ is nonzero, and hence we can divide by the order $\# \Gamma$ of
$\Gamma$. From \eqref{e26} it follows immediately that
$$
\textbf{q}\circ\widehat{\eta}\,\,=\,\, {\rm Id}_{(\phi^* (f_*{\mathcal O}_X)_*)/\phi^*S_*}.
$$
Note that ${\mathcal I}\,:=\, {\rm image}(\widehat{\eta})\, \subset\, \phi^* (f_*{\mathcal O}_X)_*$
is preserved by the action of $\Gamma$ on $\phi^* (f_*{\mathcal O}_X)_*$. Since $\phi$ is
tamely ramified, this implies that there is a unique parabolic subbundle
$$
F_*\,\, \subset\, (f_*{\mathcal O}_X)_*
$$
such that $\phi^*F_*\,\,=\,\, {\mathcal I}$. Moreover, we have
\begin{equation}\label{e27}
(f_*{\mathcal O}_X)_*\,\,=\,\, F_*\oplus S_*,
\end{equation}
because $\phi^*(f_*{\mathcal O}_X)_*\,\,=\,\, \phi^* F_*\oplus \phi^*S_*$.

If $F_*$ is nonzero, then \eqref{e27} contradicts the fact that $S_*$ is the unique maximal parabolic
polystable subsheaf (in other words, the socle) of $(f_*{\mathcal O}_X)_*$ of parabolic degree zero. Indeed, if
$W_*\, \subset\, F_*$ is a parabolic polystable subbundle of parabolic degree zero, then $W_*\oplus S_*$ is
also a parabolic polystable subbundle of $(f_*{\mathcal O}_X)_*$ of parabolic degree zero.
Therefore, we have $F_*\,=\, 0$, which implies that $(f_*{\mathcal O}_X)_*\,=\, S_*$, or
in other words, $(f_*{\mathcal O}_X)_*$ is polystable.
\end{proof}

\begin{theorem}\label{thm1}
The parabolic vector bundle $E_*$ in \eqref{e14} is parabolic polystable of parabolic degree zero.
\end{theorem}

\begin{proof}
We know from Proposition \ref{prop1} and Lemma \ref{lem1} that the parabolic vector bundle
$(f_*{\mathcal O}_X)_*$ is parabolic polystable of parabolic degree zero. It was observed in the last
paragraph of Section \ref{se2.3} that the parabolic subbundle ${\mathcal O}_Y\, \subset\, f_*{\mathcal O}_X$
in \eqref{e14} does not have any nonzero parabolic weight, and we have
$\text{degree}({\mathcal O}_Y)\,=\, 0$. From these it follows that the
parabolic subbundle $E_*\,\, \subset\,\,(f_*{\mathcal O}_X)_*$ is parabolic polystable of parabolic
degree zero.
\end{proof}

\section{Parabolic stability}

\begin{lemma}\label{lem2}
Take any $f$ as in \eqref{e1} such that there is an irreducible smooth projective curve $Z$ and
morphisms
$$
\varpi\,:\, X\, \longrightarrow\, Z \ \ \, \text{ and }\ \ \, \varphi\,:\, Z\, \longrightarrow\, Y
$$
for which $f\,=\, \varphi\circ\varpi$ and ${\rm degree}(\varphi)\, <\,{\rm degree}(f)$. Then the
parabolic vector bundle $E_*$ in \eqref{e14} is not stable.
\end{lemma}

\begin{proof}
The given condition that $f\,=\, \varphi\circ\varpi$ implies that the parabolic vector bundle $(\varphi_*
{\mathcal O}_Z)_*$ is a parabolic subbundle of $(f_*{\mathcal O}_X)_*$. Lemma \ref{lem1} and
Proposition \ref{prop1} together say that both
$(\varphi_* {\mathcal O}_Z)_*$ and $(f_*{\mathcal O}_X)_*$ are parabolic polystable of
parabolic degree zero. Consequently, the image of the composition of maps 
$$(\varphi_* {\mathcal O}_Z)_*\, \hookrightarrow\, (f_*{\mathcal O}_X)_*\,\longrightarrow\, E_*,$$
where $(f_*{\mathcal O}_X)_*\,\longrightarrow\, E_*$ is the projection for the decomposition
in \eqref{e14}, violates the stability condition for $E_*$.
\end{proof}

Take any $f$ as in \eqref{e1}. As in \eqref{e18}
$$
\phi\,\, :\,\, M\,\, \longrightarrow\,\, Y
$$
is the Galois closure of $f$. As before, let $\Gamma$ denote the Galois group ${\rm Gal}(\phi)$ and
$\delta\, :\, M\, \longrightarrow\, X$ be as in \eqref{e19}. As noted before, $\delta$ is a Galois covering.
Let $H$ denote the Galois group ${\rm Gal}(\delta)$. Then
\begin{equation}\label{eh}
H \,\, \subset\,\, \Gamma.
\end{equation}
Let
\begin{equation}\label{eh2}
{\mathbb V}\,\, \subset\,\, k[\Gamma/H]
\end{equation}
be the subspace consisting of all $\sum_{\gamma\in\Gamma/H} c_\gamma\cdot \gamma\, \in\, k[\Gamma/H]$,
where $c_\gamma\, \in\, k$, such that
$$
\sum_{\gamma\in \Gamma/H} c_\gamma\,\,=\,\, 0.
$$
The left-translation action of $\Gamma$ on $\Gamma/H$ produces an action of $\Gamma$ on $\mathbb V$.

\begin{proposition}\label{prop2}
If the above action of $\Gamma$ on $\mathbb V$ is irreducible, then the parabolic vector bundle
$E_*$ in \eqref{e14} is parabolic stable.
\end{proposition}

\begin{proof}
In the proof of Proposition \ref{prop1} it was shown that $\phi^*(f_*{\mathcal O}_X)_*$ is a
subbundle of $\phi^*(\phi_*{\mathcal O}_M)_*\,=\, {\mathcal O}_M\otimes_k k[\Gamma]$ of degree
zero (see \eqref{e22} and \eqref{e17}); recall that $\phi^*(\phi_*{\mathcal O}_M)_*$ does not
have any nonzero parabolic weight. This immediately implies that there is a unique subspace
$${\mathbb W}\, \subset\, k[\Gamma]$$ such that the subbundle $\phi^*(f_*{\mathcal O}_X)_*\, \subset\,
{\mathcal O}_M\otimes_k k[\Gamma]$ coincides with ${\mathcal O}_M\otimes_k {\mathbb W}\, \subset\,
{\mathcal O}_M\otimes_k k[\Gamma]$. Now it is straight-forward to see that
${\mathbb W}\,=\, k[\Gamma/H]$, where $H$ is the subgroup in \eqref{eh}. Consequently, 
the parabolic vector bundle $E_*$ in \eqref{e14} has the following property:
\begin{equation}\label{e28}
\phi^*E_*\,=\, {\mathcal O}_M\otimes_k {\mathbb V}\, \subset\,
{\mathcal O}_M\otimes_k k[\Gamma]\,=\, \phi^*(\phi_*{\mathcal O}_M)_*,
\end{equation}
where $\mathbb V$ is the subspace in \eqref{eh2}.

Assume that the parabolic vector bundle $E_*$ is not parabolic stable. Let
$$0\,\, \neq\,\, F_*\,\, \subsetneq\,\, E_*$$ be a parabolic subbundle that violates the stability 
condition for $E_*$. Since $E_*$ is parabolic polystable (see Theorem \ref{thm1}), this implies that
$$
\text{par-deg}(F_*)\,\,=\,\, 0.
$$
Hence from \eqref{e17a} it follows that $\text{par-deg}(\phi^* F_*)\, =\, 0$. Since $\phi^*F_*$ is
a parabolic subbundle, of parabolic degree zero, of $\phi^*(\phi_*{\mathcal O}_M)_*\,=\,
{\mathcal O}_M\otimes_k k[\Gamma]$, there is a unique subspace
${\mathbb W}_0\, \subset\, k[\Gamma]$ such that the subbundle $\phi^*F_*\, \subset\,
{\mathcal O}_M\otimes_k k[\Gamma]$ coincides with ${\mathcal O}_M\otimes_k {\mathbb W}_0\, \subset\,
{\mathcal O}_M\otimes_k k[\Gamma]$.

We have
$$
0\, \, \not=\,\, {\mathbb W}_0\,\,\subset\,\, {\mathbb V}
$$
because $F_*\, \subset\, E_*$ (see \eqref{e28}). Also
${\mathbb W}_0$ is preserved by the action of $\Gamma$ on $k[\Gamma]$ because $\phi^*F_*$ is pulled
back from $Y$. Hence the $\Gamma$--module $\mathbb V$ is not irreducible. In view of this contradiction
we conclude that $E_*$ is parabolic stable.
\end{proof}

\begin{theorem}\label{thm2}
If the Galois group $\Gamma$ is the symmetric group
$S_r$ or the alternating group $A_r$, then $E_*$ in \eqref{e14} is parabolic stable.
\end{theorem}

\begin{proof}
If $\Gamma$ is $S_r$ then $H\,=\,\{\sigma \,\in\, S_r\,\,\big\vert\,\, \sigma(r)\,=\,r\}
\,\cong\, S_{r-1}$. Moreover, the representation $\mathbb V$ 
is the standard representation of $S_r$ which is irreducible. Even if $\Gamma$ is $A_r$, the action of 
$\Gamma$ on $k(\Gamma/H)$ is same as the restriction of $S_r$-action on $k(S_r/S_{r-1})$ to the subgroup $A_r$. 
Hence in this case as well $\mathbb V$ is the restriction of the standard representation to $A_r$, which again is 
irreducible.
 
 Now the result follows from Proposition \ref{prop2}.
\end{proof}

\begin{corollary}\label{cor1}
Let $f$ of \eqref{e1} be genuinely ramified and Morse, then $E_*$ in \eqref{e14} is parabolic stable.
\end{corollary}

\begin{proof}
The main theorem of \cite{BKP} implies that the monodromy $\Gamma$ of $f$ is $S_r$. Hence the result
follows from Theorem \ref{thm2}.
\end{proof}

\section*{Acknowledgements}

We thank the referee for helpful comments. The first-named author is partially supported by a J. C. Bose Fellowship (JBR/2023/000003).


\end{document}